%% file: 5dRevision.tex
\documentclass[reqno,11pt]{amsart}  
\usepackage[utf8]{inputenc}

\input{packages}
\input{commands}

\newcommand{\lra}{\longleftrightarrow}
\newcommand{\lro}[1]{\xleftrightarrow{\,\,{#1}\,\,}}

\newcommand{\dfrak}{\mathfrak{d}}
\newcommand{\nminus}{\overline{n}}
\newcommand{\dist}{\mathrm{dist}}
\renewcommand{\Cap}{\mathrm{Cap}}
\newcommand{\Es}{\mathrm{Es}}

\title{1+O(1) asymptotics for loop percolation in five and higher dimensions}
\author{Quirin Vogel$^{1}$}
\date{August 2025}

\begin{document}

\maketitle
\begin{abstract}
We calculate the one-arm probability and the two-point function for loop percolation in dimensions five and higher on the lattice to first order. This answers a question posed by Y. Chang and A. Sapozhnikov in \textit{Probability Theory and Related Fields} (2016), 164:979–1025.
\end{abstract}

\centerline{\textit{$^1$Alpen--Adria--Universität Klagenfurt,
Department of Statistics, Klagenfurt, Austria}}
\centerline{\textsc{quirin.vogel@aau.at}}
\bigskip


\section{Introduction}\label{sec:intro}
Loop percolation is a beautiful long-range percolation model on the lattice\footnote{This article is restricted to $\Z^d$, for simplicity. However, it suffices to be on a countable graph with finite Green function for this model to be interesting, see \cite{le2013markovian,chang2016phase}.} $\Z^d$. It was first introduced in \cite{le2013markovian}. It arises naturally when studying the Symanzik/Le Jan loop representation of the Gaussian free field. More precisely, loop percolation is closely related to the Gaussian free field and its level sets, see \cite{werner2016spatial,10.1214/15-AOP1019,lawler2018topics}.

The model can be described as follows: fix an activity $\alpha>0$. Then,
\begin{itemize}
    \item at each vertex $x\in\Z^d$ independently sample a Poisson random variable $V_x\in\N\cup\gk{0}$ with intensity $\alpha c_d$, for some constant\footnote{$c_d=\log\rk{G(0,0)}$, where $G$ is the random walk Green's function.} $c_d>0$.
    \item For each $V_x>0$, sample lengths independently $\ell_1,\ldots,\ell_{V_x}$ where $\P\rk{\ell=k}\propto k^{-1}p_k(0,0)$ and $p_k(0,0)$ is the probability that the simple random walk returns to the origin in $k$ steps.
    \item Independently sample random walk bridges $\omega_1,\ldots, \omega_{V_x}$ according to the laws $\B_{x,x}\hk{\ell_1},\ldots, \B_{x,x}\hk{\ell_{V_x}}$, where $\B_{x,x}\hk{\ell}$ is the random walk bridge measure from $x$ to $x$ in time $\ell$.
    \item Declare an edge $e$ to be \textit{open} if it is traversed by at least one random walk. Let $\Ccal$ be the set of vertices contained in maximal connected components of open edges.
\end{itemize}
See Figure~\ref{fig:placeholder} for a visualization. Denote by $\P=\P_\alpha$ the loop soup sampled as above. Loop percolation studies the connectivity of the set of open edges, in particular the cluster located at the origin.
\begin{figure}
    \begin{minipage}{.48\textwidth}
                  \centering
                  \includegraphics[width=0.75\linewidth]{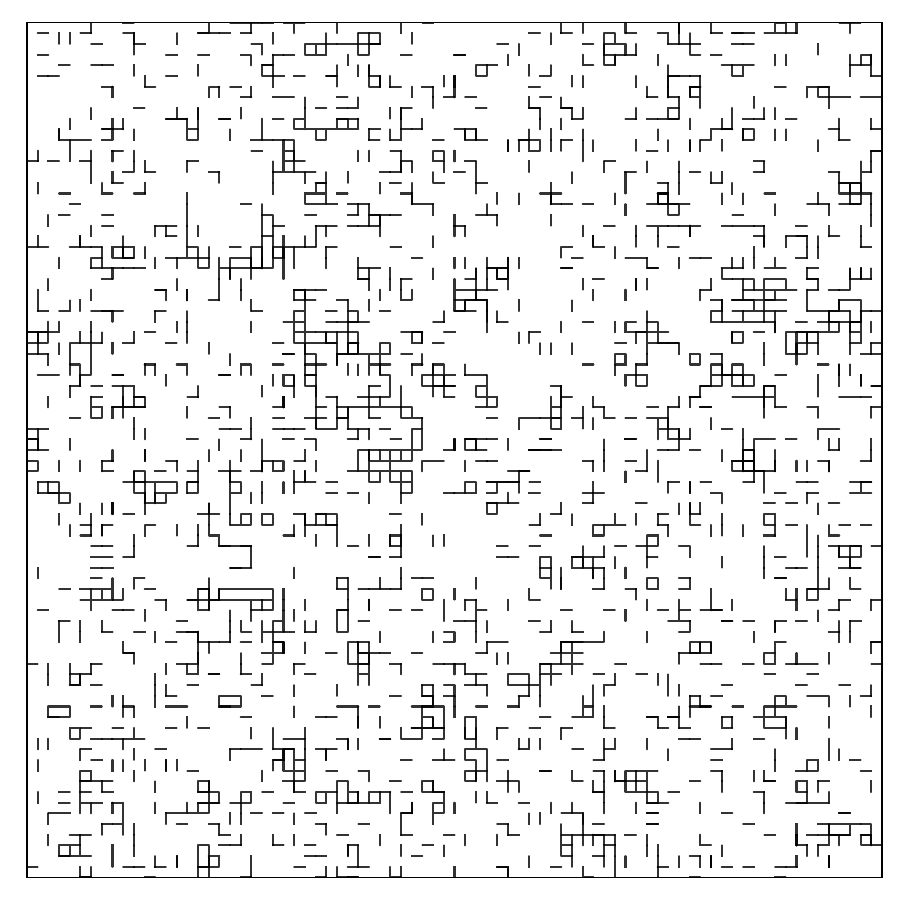}
 \end{minipage}
     \begin{minipage}{.48\textwidth}
                  \centering
                  \includegraphics[width=0.75\linewidth]{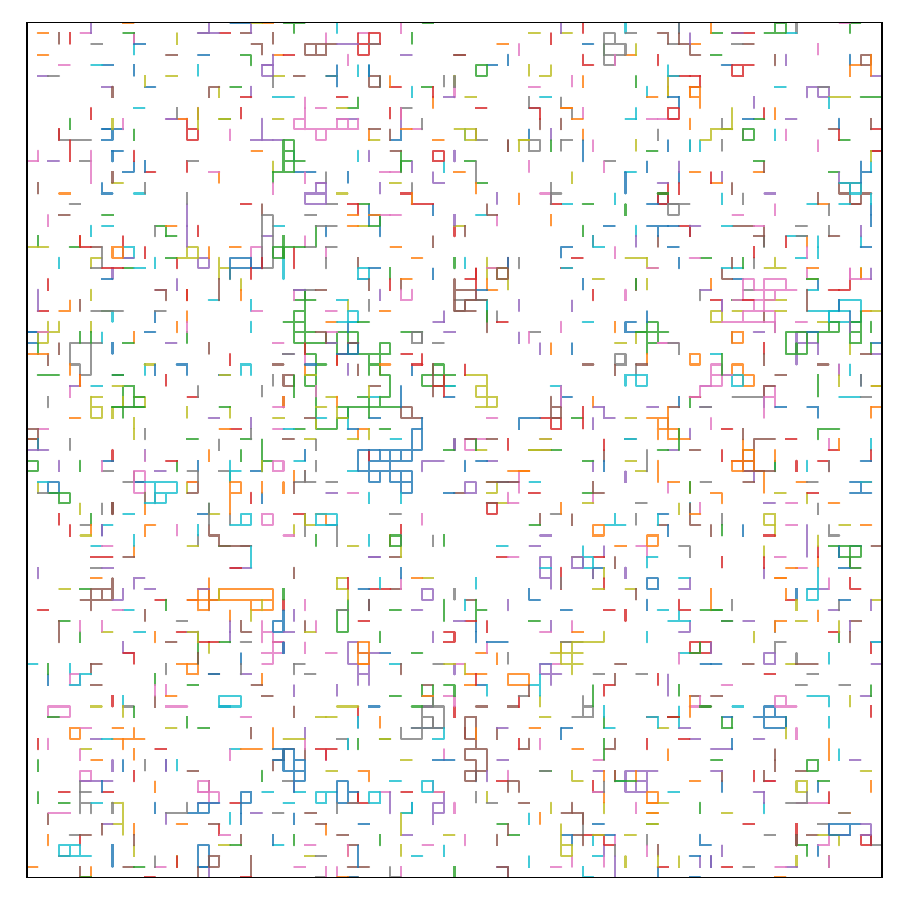}
 \end{minipage}
 \caption{The open edges for a sample of loop percolation on the left. On the right the different loops are colored differently. Figure from a simulation of loop percolation for $d=3$, projected onto the first two coordinates.}
    \label{fig:placeholder}
\end{figure}

\textbf{Our contribution}. The purpose of this article is to study subcritical loop percolation in dimensions five and higher. Define $\alpha_\#>0$ as the supremum over all $\alpha>0$ for which the expected size of the cluster at the origin is finite. It always holds that $\alpha_\#\le \alpha_c$.

Let $\gk{0\lra \partial B_n}$ denote the event that the origin is connected to the boundary of the centered ball with radius $n$ through open edges. The first result of this paper is the existence of the limit
\begin{equation}\label{eq:intro1}
    \lim_{n\to\infty}n^{d-2}\P\rk{0\lra \partial B_n}\, ,
\end{equation}
for all $\alpha<\alpha_\#$. We also show that
\begin{equation}\label{eq:intro2}
    \lim_{\abs{x}\to\infty} \abs{x}^{2d-4}\P\rk{0\lra  x}\, ,
\end{equation}
exists, where $\gk{0\lra  x}$ denotes the existence of an open path from the origin to $x\in\Z^d$. The value of these limits is also computed and it is shown that the conditional law of the process can be realized by adding a single long loop to the soup.

In \cite{chang2016phase}, the authors showed that given $d\ge 5$, one has that $0<\alpha_\#$ and that for $\alpha<\alpha_\#$, it holds that
\begin{equation}\label{eq:CSEstimate}
    C_1\alpha n^{2-d}\le \P\rk{0\lra \partial B_n} \le C_2(\alpha)n^{2-d}\, .
\end{equation}
They also proved upper and lower bounds on $\P\rk{0\lra  x}$. In their list of open problems (see \cite[Section 8]{chang2016phase}), they include the question of the existence of limits in Eq.~\eqref{eq:intro1} and Eq.~\eqref{eq:intro2}.

Our work was inspired by \cite{duminil2020subcritical}, where it was shown that for Poisson--Boolean percolation with polynomial tails in the subcritical regime, connecting the origin to the complement of the centered ball $B_n$ with radius $n$, has the same price as facilitating this connection with a single mark (asymptotically in $n$). The same cannot be true for loop percolation, because the trajectories of random walks are less ``dense"; they are two-dimensional objects in $d\ge 5$ dimensional space. We show that for $d\ge 5$ and $\alpha$ small enough, the event that the origin is connected to $\partial B_n$ is typically realized as follows: sample a subcritical loop soup. To its (highly localized) cluster at the origin, attach (harmonically) a long loop $\omega$ that intersects $\partial B_n$. The loop $\omega$ may intersect the origin, but with positive probability it does not. We furthermore prove that a similar principle applies to the two-point function: to connect the origin to a faraway $x$, first independently sample clusters at $0$ and at $x$. Then, add a big loop connecting the two clusters. In the proof explicit error bounds depending on $n$ or $\abs{x}$ are derived.\\
Our main tools are refined loop/percolation estimates together with the use of the Mecke equation, which has not been applied in this context to the best of our knowledge. The Mecke equation allows us to separate the law conditioned to connect into an unconditioned law decorated with a long loop. We expect this behavior to hold for a variety of long-range subcritical percolation models based on Poisson point processes, and hence expect that our method is useful to other cases. We remark that the Mecke equation is key for deriving Russo's formula in the context of Poissonian models, see \cite{duminil2020subcritical}.

We emphasize the importance of \cite{chang2016phase}. In this work, the authors showed that for $d\ge 3$, the critical percolation threshold $\alpha_c$ is nontrivial. The proof of the result is not easy because the one-arm probability decays polynomially in the subcritical phase. Moreover, the authors showed that in dimensions three and four, the expected cluster size at the origin is infinite for every $\alpha>0$, hence $0=\alpha_\#<\alpha_c$. They also proved that the strategy of adding a single long loop to the soup cannot work for $d=3,4$.

\textbf{Open questions}: for $d\ge 5$, it remains to show that $\alpha_c=\alpha_\#$ and that hence our asymptotics hold for the entire subcritical domain. This problem is currently under investigation for sufficiently high dimensions. Furthermore, it also remains open whether for $d=4$, the one-arm connection decays as $n^{-2}$ \textit{with logarithmic corrections}. In fact, it can be shown that $\P\rk{0\lra \partial B_n}\gg n^{-2}$. The difference from $d\ge 5$ is that for $d=4$, the capacity of the random walk range lives on a different scale, see Remark~\ref{rem:remark}. 

Before stating the main result of this article, we mention some important works in loop percolation: in \cite{10.1214/15-AOP1019} it was shown that $\alpha_c\ge \frac{1}{2}$ for a wide variety of graphs, and in \cite{10.1214/16-ECP4571} it was shown that $\alpha_c=\frac{1}{2}$ for the half-plane. Chemical distances were studied in \cite{ding2018chemical} for the planar case. The supercritical regime for loop percolation was studied in \cite{chang2017supercritical}. For loop percolation on \textit{metric} graphs, much progress has been achieved recently: for a large class of transitive graphs, it has been shown that $\alpha_c=1/2$, see \cite{chang2024percolation}, based on \cite{10.1214/15-AOP1019,drewitz2022cluster}. The loop soup in the context of the reinforced random walk was studied in \cite{chang2025reinforced}. See \cite{alves2019decoupling} for results on the percolation of the \textit{vacant set} of the loop soup and \cite{rath2022percolation} for worm percolation (which are open loops!).
\section{Results}
The measure $\P=\P_\alpha$ is defined in the Introduction; an alternative construction is given in Section~\ref{sec:bacl}. Write $\Ccal_0$ for the cluster of vertices connected to the origin via open edges \textit{including the origin itself}, so that $\Ccal_0$ always contains the origin. Define the relevant critical parameter $\alpha_\#$ as
\begin{equation}
    \alpha_\#=\sup\gk{\alpha>0\colon \E_\alpha\ek{\#\Ccal_0}<\infty}\le \alpha_c\, ,
\end{equation}
which is strictly positive for $d\ge 5$ by \cite{chang2016phase}. Write $ C_d$ for the leading-order constant\footnote{$G(0,x)=C_d\abs{x}^{2-d}(1+o(1))$ as $\abs{x}\to \infty$, see \cite[Section 4.3]{lawler2010random}.} of the simple random walk Green's function in $d\ge 3$ 
\begin{equation}\label{eq:defGreenConstant}
     C_d=\frac{d\Gamma(d/2)}{(d-2)\pi^{d/2}}\, .
\end{equation}
Recall that the (random-walk) capacity of a set $A\subset\Z^d$ is given by
\begin{equation}
    \Cap\rk{A}=\sum_{x\in A}\P_x\rk{\omega(n)\notin A\textnormal{ for all }n\ge 1}\, ,
\end{equation}
for $\P_x$ the law of a simple random walk.

Our first result concerns the one-arm probability:
\begin{theorem}\label{thm1}
    Recall that $\Ccal_0$ is the open cluster at the origin and always contains the origin itself. For all $\alpha<\alpha_\#$ and $d\ge 5$, as $n\to\infty$
    \begin{equation}\label{eq:thm1eq1}
        \P_\alpha\rk{0\lra \partial B_n}=\alpha\frac{C_d\E_\alpha\ek{\Cap\rk{\Ccal_0}}}{n^{d-2}}\rk{1+o(1)}\, .
    \end{equation}
Furthermore
    \begin{equation}\label{eq:thm1eq2}
        \lim_{k\to\infty}\lim_{n\to\infty }  \P_\alpha\rk{\exists \omega\colon \Ccal_0\rk{\eta\setminus \omega}\subseteq B_k\textnormal{ and }\Ccal_0\rk{\eta\setminus \omega}\lro{\omega}\partial B_n\mid 0\lra \partial B_n}=1\, .
    \end{equation}
\end{theorem}
\begin{remark}
    Unfortunately, for $d=4$, $\E\ek{\Cap\rk{\Ccal_0}}$ can be lower bounded by $\sum_{n\ge 1}\frac{1}{n\log(n)}$ and diverges logarithmically. For $d\ge 5$, $\E_\alpha\ek{\Cap\rk{\Ccal_0}}$ can be calculated to arbitrary precision.
\end{remark}
The following theorem gives our result for the two-point function.
\begin{theorem}\label{thm:twoPoint}
For all $\alpha<\alpha_\#$, as $\abs{x}\to\infty$
  \begin{equation}
        \P_\alpha\rk{0\lra x}=\alpha \frac{C_d^2\E_\alpha\ek{\Cap\rk{\Ccal_0}}^2}{\abs{x}^{2d-4}}\rk{1+o(1)}\, .
  \end{equation}
\end{theorem}
\begin{remark}
\begin{enumerate}
\item The square-relation between one-point and two-point function satisfies the critical hyperscaling relation $d\rho=\delta+1$ (see \cite[Chapter 9]{grimmett1999percolation} for further background), which is expected to hold only for $d\le d_{\mathrm{c}}$. Theorem~\ref{thm1} and Theorem~\ref{thm:twoPoint} are only proven in the strictly subcritical domain, hence this connection is merely suggestive.
    \item 
    The above can be generalized so that for $K,L\subset \Z^d$ fixed, $\P\rk{K\lra L}$ is given by $\alpha C_d^2 \E\ek{\Cap(\Ccal\cap K)}\E\ek{\Cap(\Ccal\cap L)}\dist(K,L)^{4-2d}$ up to first order. It should also continue to hold if $K,L$ have diameters that grow more slowly than their distance, see Lemma~\ref{lem:ConnectTwoSetsExact}. In this article, we consider only the two-point function and leave the extension of our method to future work.
\end{enumerate}
\end{remark}
\section{Proof}
\subsection{Background}\label{sec:bacl}
It is convenient to describe loop percolation using a Poisson point process. For this, set 
\begin{equation}
    \Omega=\bigcup_{n\ge 1}\gk{f\colon \gk{0,\ldots n}\to \Z^d\textnormal{ with }\abs{f(m)-f(m+1)}=1\textnormal{ for all }m\in \gk{0,\ldots, n-1}}\, ,
\end{equation}
the space of all finite nearest neighbor trajectories (here, $\abs{x}$ is the standard Euclidean norm). For $\omega\in \Omega$, let $\ell(\omega)$ be the duration of $\omega$, i.e., $\ell(\omega)=n$ if and only if $\omega\colon\gk{0,\ldots,n}\to\Z^d$. Define then for $\omega$ with $\ell(\omega)=n$, the loop weight $\mu(\omega)$ as
\begin{equation}
    \mu(\omega)=\frac{1}{n}\rk{\frac{1}{2d}}^{n}\1\gk{\omega(0)=\omega(n)}\, .
\end{equation}
The measure $\mu$ is a $\sigma$-finite measure on $\Omega$. Alternatively, $\mu$ can be defined by $\mu=\sum_{x\in\Z^d}\sum_{n\ge 1}\frac{1}{n}\P_{x,x}^n$ where $\P_{x,x}^n(A)=\P_x\rk{A\cap\gk{\omega(n)=x}}$ is the unnormalized bridge measure of a simple random walk.

Define $\P=\P_\alpha$, the measure introduced in the Section~\ref{sec:intro}, as the Poisson point process (see \cite{last2018lectures} for a reference on Poisson point processes) with intensity measure $\alpha\mu$. Formally, we have that $\P_\alpha=\ex^{-\alpha\mu[\Omega]}\sum_{n\ge 1}\frac{\alpha^n\mu^{\otimes n}}{n!}$. A sample $\eta$ from $\P$ can then either be described as a random point measure or as a collection of loops. Denote $\eta$ as
\begin{equation}
    \eta=\sum_{\omega\in \eta}\delta_\omega=\gk{\omega\in \Omega \colon \omega \in \eta}\, .
\end{equation}
Let us mention a subtle point: as a loop $\omega$ can appear multiple times in $\eta$, a multiset should be used to describe $\eta$ (i.e., $\P$ is not simple in point-process terminology). For loop percolation, this can be ignored, since sampling a loop twice does not change connectivity. Formally, we can avoid the issue of multiple loop occurrences by considering the Poisson point process on \textit{continuous-time random walk bridges}. This approach is straightforward because the waiting times do not alter the underlying connectivity; for details on the loop measure for continuous-time random walk bridges, see \cite{le2024random}.

Write $\omega\in\eta$ if the loop $\omega$ is contained in $\eta$. Write $\eta\cup \omega=\eta+\delta_\omega$ for the configuration obtained when adding loop $\omega$ to $\eta$, similarly for $\eta\cup \zeta=\eta+\zeta$ for $\zeta$ another configuration. Write $\eta\setminus\omega$ for the configuration where the loop $\omega$ is removed.

The Poisson point process representation is convenient when working with loop percolation. It has several implications, which are listed below.
\begin{lemma}[Inequalities]
    Define an event $A$ as increasing if $\eta\in A$ satisfies $\eta+\zeta\in A$ for all $\zeta$. Then for increasing $A,B$, the \textnormal{(FKG-inequality)} holds:
    \begin{equation} \P\rk{A\cap B}\ge \P\rk{A}\P\rk{B}\, .
    \end{equation}
    Furthermore, write $A\Box B$ if $A$ and $B$ are realized through a disjoint set of loops (not edges!). We then have for $A,B$ increasing that \textnormal{(BK-inequality)}
    \begin{equation}
        \P\rk{A\Box B}\le \P\rk{A}\P\rk{B}\, .
    \end{equation}
\end{lemma}
See \cite[Theorem 20.4]{last2018lectures} for a reference of the FKG inequality for Poisson point processes and \cite[Theorem 2.3]{meester1996continuum} for the BK inequality.

Recall the Mecke equation, an important tool in the theory of (Poisson) point processes.
\begin{lemma}[The Mecke equation]\label{lem:mec}
If $\eta$ is sampled from $\P$, a Poisson point process with intensity measure $\mu$, then for $f\ge 0$ measurable
\begin{equation}
    \E\ek{\sum_{\omega\in \eta}f\rk{\eta\setminus\omega,\omega}}=\int \E\ek{f(\eta,\omega)}\d \mu(\omega)\, .
\end{equation}
\end{lemma}
See \cite[Theorem 4.5]{last2018lectures} for a reference.
\subsection{Results for the loop measure}
In this section, several useful lemmas on the loop measure $\mu$ are collected.

Define the random walk range $\Rcal$ of a loop as follows: $\Rcal(\omega)=\gk{x\in \Z^d\colon \exists k\textnormal{ with }\omega(k)=x}$. 
\begin{lemma}
Write $K\lro{\omega}L$ if $\omega$ intersects both $K,L\subset\Z^d$. If $K$ and $L$ are two finite disjoint sets, then
    \begin{equation}\label{eq:rep}
        \mu\ek{\omega\colon K\lro{\omega}L}=\sum_{x\in K}\sum_{m\ge 1}\frac{1}{m}\P_x\rk{\omega\rk{\tau_{2m}}=x}\, ,
    \end{equation}
    where $\tau_0=0$, $\tau_{2m+1}=\inf\gk{t>\tau_{2m}\colon \omega(t)\in L}$ and $\tau_{2m+2}=\inf\gk{t>\tau_{2m+1}\colon \omega(t)\in K}$, $m\ge 0$. Furthermore,
    \begin{equation}\label{eq:rangeRep}
        \mu\ek{\#\Rcal(\omega)\textnormal{ with } K\lro{\omega}L}=\sum_{x\in K}\sum_{m\ge 1}\frac{1}{m}\E_x\ek{\#\Rcal\rk{\omega\rk{\tau_{2m}}},\, \omega\rk{\tau_{2m}}=x}\, .
    \end{equation}
\end{lemma}
where $\Rcal\rk{\omega(\tau_{2m})}$ is the range up to time $\tau_{2m}$.\begin{proof}
 Eq.~\eqref{eq:rep} was shown in \cite{chang2016phase}, see the second equation in the proof of their Lemma 2.7. Eq.~\eqref{eq:rangeRep} follows from \cite[Eq. (6)]{chang2016phase}, as this gives for $M\in\N$
 \begin{equation}
     M\mu\ek{\1\gk{\#\Rcal(\omega)=k}\textnormal{ with }\colon K\lro{\omega}L}=M\sum_{x\in K}\sum_{m\ge 1}\frac{1}{m}\E_x\ek{\1\gk{\#\Rcal\rk{\omega\rk{\tau_{2m}}}=M},\, \omega\rk{\tau_{2m}}=x}\, .
 \end{equation}
 Summing over $M\in\N$ yields the claim.
\end{proof}

Write $B_k=\gk{x\in\Z^d\colon \abs{x}\le k}$ and $B_k(y)=\gk{x\in\Z^d\colon \abs{y-x}\le k}$ for the ball centered around the origin, respectively, around $y$. For $K\subset \Z^d$, write $\diam(K)=\sup_{x,y\in K}\abs{x-y}$, write $\dist(A,B)=\inf_{\mycom{x\in A}{y\in B}}\abs{x-y}$, and abbreviate $\dist(x,A)=\dist\rk{\gk{x},A}$. For $K\subset \Z^d$, set $\partial K=\gk{x\in K\colon \dist(x,K^c)=1}$ the (inner) boundary of the set $K$. 

  \begin{lemma}\label{lem:singleLoopExact}
        Suppose $d\ge 3$. For all $K$ with $\diam(K)=o\rk{n}$ as $n\to\infty$
        \begin{equation}
            \mu\ek{\exists \omega\colon K\lro{\omega}\partial B_n}= C_d \Cap(K)n^{2-d}(1+o(1))\, ,
        \end{equation}
        where $C_d$ was defined in Eq.~\eqref{eq:defGreenConstant}.
    \end{lemma}
    \begin{proof}
        By Eq.~\eqref{eq:rep}
        \begin{equation}\label{eq:81120252}
            \mu\ek{\exists \omega\colon K\lro{\omega}\partial B_n}=\sum_{x\in K}\sum_{m\ge 1}\frac{1}{m}\P_x\rk{\omega(\tau_{2m})=x}\, ,
        \end{equation}
        where $\tau_0=0$, $\tau_{2m+1}=\inf\gk{t>\tau_{2m}\colon \omega(t)\in \partial B_n}$ and $\tau_{2m+2}=\inf\gk{t>\tau_{2m+1} \colon \omega(t)\in \partial K}$, $m\ge 0$. Write $H_K$ for the hitting time of a set $K\subset \Z^d$, i.e., $H_K=\inf\gk{t\ge 0\colon \omega\in K}$. By \cite[Proposition 6.5.1]{lawler2010random}, for every $y\in \partial B_n$ as $n\to\infty$
        \begin{equation}\label{eq:81120251}
            \P_y\rk{H_K<\infty}=\frac{C_d\Cap(K)}{n^{d-2}}\rk{1+o(1)}\, .
        \end{equation}
        The monotonicity of the capacity implies that $\Cap(K)\le \Cap\rk{B_{\diam(K)}}=o\rk{n^{d-2}}$. Hence, Eq.\eqref{eq:81120251} is $o(1)$ for all $K$ conforming to the assumption $\diam(K)=o\rk{n}$.
        
        By \cite[p. 990]{chang2016phase}, there exists $C>0$ such that for all $m\ge 1$
        \begin{equation}
            \P_x\rk{\omega(\tau_{2m})=x}\le C \max_{x\in \partial K}\P_x\rk{\omega(\tau_{2})=x}\rk{\max_{y\in \partial B_n}\P_y\rk{H_K<\infty}}^{m-1}\, .
        \end{equation}
    By the assumption on $K$ and by Eq.~\eqref{eq:81120251}, the exponentiated term is uniformly $o(1)$ and hence Eq.~\eqref{eq:81120252} is dominated by the first term, i.e., for every $x\in K$
    \begin{equation}\label{eq:81120253}
            \sum_{m\ge 1}\frac{1}{m}\P_x\rk{\omega(\tau_{2m})=x}=\P_x\rk{\omega(\tau_{2})=x}\rk{1+o(1)}\, .
    \end{equation}
        However, by Eq.~\eqref{eq:81120251} and \cite[Proposition 6.5.4]{lawler2010random}
        \begin{equation}
            \P_y\rk{\omega\rk{H_K}=x}=\frac{C_d\Es_{K}(x)}{n^{d-2}}\rk{1+o(1)}\, ,
        \end{equation}
        where $\Es_K(x)=\P_x\rk{\forall m\ge 1: \omega(m)\notin K}$ is the escape probability. Hence, by the Markov property,
        \begin{equation}
            \P_x\rk{\omega(\tau_{2})=x}=\sum_{y\in\partial B_n}\P_y\rk{\omega\rk{H_K}=x}\P_x\rk{\omega\rk{H_{\partial B_n}}=y}=\frac{C_d\Es_{K}(x)}{n^{d-2}}\rk{1+o(1)}\, .
        \end{equation}
        The proof then follows from recalling that $\sum_{x\in K}\Es_K(x)=\Cap(K)$.
    \end{proof}
The next lemma computes the asymptotic probability of connecting two ``small" sets.
\begin{lemma}\label{lem:ConnectTwoSetsExact}
    For $d\ge 3$. Suppose that $K, L\subset\Z^d$ satisfy $\diam(K)+\diam(L)=o\rk{\dist(K,L)}$. We then have
    \begin{equation}
        \mu\ek{\omega\colon K\lro{\omega}L}= \frac{C_d^2\Cap(L)\Cap(K)}{\dist(K,L)^{2d-4}}\rk{1+o(1)}\, .
    \end{equation}
\end{lemma}
Note that, unlike the setting of the previous lemma, the path $\omega$ originating at $K$ avoids $L$ with high probability.
\begin{proof}
    Define  $\tau_0=0$, $\tau_{2m+1}=\inf\gk{t>\tau_{2m}\colon \omega(t)\in L}$ and $\tau_{2m+2}=\inf\gk{t>\tau_{2m+1}\colon \omega(t)\in K}$, $m\ge 0$. Following the proof of Lemma~\ref{lem:singleLoopExact}, we first compute
    \begin{equation}\label{eq:8820251}
        \sum_{x\in K}\P_x\rk{\omega(\tau_2)=x}=\sum_{x\in K}\P_x\rk{H_L<\infty}\sum_{y\in L}\P_x\rk{\omega(H_L)=y|H_L<\infty}\P_y\rk{\omega(H_K)=x}\, .
    \end{equation}
    Since $L$ and $K$ are sufficiently far apart, $\P_y\rk{\omega(H_K)=x}$ does not depend on $y$ (up to first order, again \cite[Proposition 6.5.1 and 6.5.4]{lawler2010random}). Hence, the leading order term in Eq.~\eqref{eq:8820251} is given by
    \begin{equation}
         \sum_{x\in K}\frac{C_d\Cap(L)}{\dist(K,L)^{d-2}}\frac{C_d\Es_K(x)}{\dist(L,K)^{d-2}}=\frac{C_d^2\Cap(L)\Cap(K)}{\dist(K,L)^{2d-4}}\, .
    \end{equation}
    As in Eq.~\eqref{eq:81120253}, the terms with $m\ge 2$ in the expansion given in Eq.~\eqref{eq:rep} are negligible. This completes the proof.
\end{proof}
    \begin{lemma}\label{ref:lemmasize}
        Fix $d\ge 3$. As $m\to\infty$, it holds that
        \begin{equation}
            \mu\ek{\#\Rcal(\omega)|\diam(\omega)>m,0\in\omega}=\Ocal\rk{m^{2}}\, .
        \end{equation}
    \end{lemma}
    \begin{proof}
        Recall that by Lemma~\ref{lem:singleLoopExact} for some $C>0$
        \begin{equation}\label{eq63020252}
             \mu\ek{\diam(\omega)>m,0\in\omega}\ge C m^{2-d}\, .
        \end{equation}
        Define $\tau_0=0$, $\tau_{2n+1}=\inf\gk{t>\tau_{2n}\colon \omega(t)\in \partial B_m}$ and $\tau_{2n+2}=\inf\gk{t>\tau_{2n+1}\colon \omega(t)=0}$, $n\ge 0$. By Eq.~\eqref{eq:rangeRep}
        \begin{equation}\label{Eq.630205}
            \mu\ek{\#\Rcal(\omega),\diam(\omega)>m,0\in\omega}=\sum_{n\ge 1}\frac{1}{n}\E_0\ek{\#\Rcal\rk{\omega(\tau_{2n})},\diam(\omega)>m,\tau_{2n}<\infty}\, ,
        \end{equation}
        where $\Rcal\rk{\omega(\tau_{2n})}$ is the range up to time $\tau_{2n}$. Using the bound $\#\Rcal\rk{\omega}\le \ell(\omega)$, estimate
        \begin{equation}
            \E_0\ek{\#\Rcal\rk{\omega(\tau_{2n})},\diam(\omega)>m,\tau_{2n}<\infty}\le \E_0\ek{\tau_{2n},\diam(\omega)>m,\tau_{2n}<\infty}\, .
        \end{equation}
        We first analyze the case $n=1$ in Eq.~\eqref{Eq.630205}: by the strong Markov property, $\tau_2$ can be additively decomposed into the time it takes to hit $\partial{B_m}$ (which is $\tau_1$) and the time it takes to return to zero. This gives
        \begin{equation}\label{eq:332026}
            \E_0\ek{\tau_{2},\diam(\omega)>m,\tau_{2}<\infty}\le C\E_0[H_{\partial B_m}]\P_0\rk{\tau_2<\infty}+\sum_{y\in\partial B_m}\E_y\ek{H_0,H_0<\infty}\P_0\rk{\omega\rk{H_{\partial B_m}}=y}\, .
        \end{equation}
        By upper-bounding the hitting time by that of a one-dimensional walker
        \begin{equation}
            \E_0[H_{\partial B_m}]\le C\E_{0,\Z^1}[H_{\partial B_m}]\le Cm^2\, .
        \end{equation}
        The probability of return to the origin of the simple random walk started at distance $m$ is given by $\Ocal\rk{m^{2-d}}$, see \cite[Proposition 6.4.2]{lawler2010random}, hence $\P_0\rk{\tau_2<\infty}=\Ocal\rk{m^{2-d}}$ and therefore $\E_0[H_{\partial B_m}]\P_0\rk{\tau_2<\infty}=\Ocal\rk{m^{4-d}}$.
        
        For the second term on the left-hand side of Eq.~\eqref{eq:332026}, compute for $y\in \partial B_m$
        \begin{equation}
            \E_{y}\ek{H_{0}, H_0<\infty}=\sum_{j\ge 1}j\P_y\rk{H_0=j}\le \sum_{j\ge 1}j\P_y\rk{\omega(j)=0}\le C\sum_{j\ge 1}j^{1-d/2}\ex^{-C m^2/j}\, ,
        \end{equation}
        where the reader can find the heat-kernel estimate in \cite{hebisch1993gaussian}. A quick calculation shows that the above is of order $\Ocal\rk{m^{4-d}}$ ($m^{2-d}$ for the power of $j$ and $m^{2}$ for the domain) uniformly in $y\in \partial B_m$. 
        
        Combining both estimates of the terms in Eq.~\eqref{eq:332026}, we conclude that
        \begin{equation}
            \E_0\ek{\tau_{2},\diam(\omega)>m,\tau_{2}<\infty}=\Ocal\rk{m^{4-d}}\, .
        \end{equation}
        To treat the case $n>1$, observe the stopping times $\rk{\tau_n}_n$ are additive. Furthermore, at each instance of $\tau_{2i+1}$, there is only a $\Ocal\rk{m^{2-d}}$ probability to return to the origin. Hence, for some $C>0$, by the strong Markov property
        \begin{multline}
            \E_0\ek{\tau_{2n},\diam(\omega)>m,\tau_{2n}<\infty}\le  Cn\rk{m^{2-d}}^{n-1}\E_0\ek{\tau_{2},\diam(\omega)>m,\tau_{2}<\infty}\\
            =\Ocal\rk{n m^{4-d}\rk{m^{2-d}}^n}\, .
        \end{multline}
    Inserting this into Eq.~\eqref{Eq.630205} yields
    \begin{equation}
         \mu\ek{\#\Rcal\rk{\omega},\diam(\omega)>m,0\in\omega}=\Ocal\rk{m^{4-d}}\, .
    \end{equation}
    Recall from Eq.~\eqref{eq63020252} that $ \mu\ek{\exists{\omega},\diam(\omega)>m,0\in\omega}\ge c\rk{m^{2-d}}$. This then concludes the proof.
    \end{proof}
    \begin{remark}\label{rem:remark}
        The bound $\Cap(K)\le \# K$ implies that $\Cap\rk{\Rcal(\omega)}\le \#\Rcal(\omega)$. If $\diam(\omega)\sim n$, $\#\Rcal(\omega)$ is of order $n^2$, see \cite{erdos1960some} for $d\ge 3$. Furthermore, $\Cap\rk{\Rcal(\omega)}$ is also of order $n^2$ for $d\ge 5$, see \cite{jain1968range}, therefore $\Cap(K)\le \# K$ is of the correct order for $d\ge 5$. For $d=4$, it holds that $\Cap\rk{\Rcal(\omega)}$ is of order $n^2/\log(n)$, see \cite{asselah2019capacity}, and hence loop percolation for $d=4$ has a different behavior.
    \end{remark}

    \begin{lemma}\label{lem:largesingkeloop}
     Suppose $d\ge 3$. As $m\to\infty$ and for $\abs{x}=o(m)$
    \begin{equation}
        \mu\ek{\omega\colon 0\lra x\textnormal{ and }\diam(\omega)>m}=\Ocal\rk{m^{2-d}\abs{x}^{2-d}}\, .
    \end{equation}
\end{lemma}
\begin{proof}
    Define for $n\ge 0$: $\tau_0=0$, $\tau_{2n+1}=\inf\gk{t>\tau_{2n}\colon \omega(t)=x}$ and $\tau_{2n+2}=\inf\gk{t>\tau_{2n+1}\colon \omega(t)=0}$. Adapting Lemma~\ref{lem:ConnectTwoSetsExact}, we can rewrite the mass of the event above as
    \begin{equation}\label{eq:81120254}
        \sum_{n\ge 1}\frac{1}{n}\P_0\rk{\omega\rk{\tau_{2n}}=0,\, H_{\partial B_m}<\tau_1}\, .
    \end{equation}
    Start with $n=1$. The random walker started in $y\in \partial B_m$ hits $x$ with a probability of
    \begin{equation}
        \P_y\rk{H_x<\infty}=\Ocal\rk{\dist(x,\partial B_m)^{2-d}}=\Ocal\rk{m^{2-d}}\, ,
    \end{equation}
    using again \cite[Proposition 6.5.1]{lawler2010random}. Furthermore, started from $x$ to hit the origin has cost of $\Ocal\rk{\abs{x}^{2-d}}$. This gives
    \begin{equation}
        \P_0\rk{\omega\rk{\tau_{2}}=x,\, H_{\partial B_m}<\tau_1}=\Ocal\rk{m^{2-d}\abs{x}^{2-d}}\, .
    \end{equation}
    As before, $n\ge 2$ can be ignored in Eq.~\eqref{eq:81120254}. This proves the lemma.
\end{proof}
\subsection{Proof of Theorem \ref{thm1}}
For the whole section, assume that $0<\alpha<\alpha_\#$ and $d\ge 5$. Write
\begin{equation}
    A\lro{\mathrm{C}}B\, ,
\end{equation}
if $A$ can be connected via open edges to $B$ with \textit{a chain of} loops $\gk{\omega\colon \mathrm{C}(\omega)=\mathrm{True}}$. For example $A\lro{\diam(\omega)\le m}B$ means a connection facilitated through loops of diameter at most $m$; $A\lro{\subseteq B_m}B$ is a connection using loops contained in $B_m$.

\begin{lemma}\label{lem:shortLoopsConnect}
    For all $\alpha<\alpha_\#$, there exists $c_1>0$ such that for all $m,n\ge 1$, it holds
    \begin{equation}
        \P\rk{0\lro{\diam(\omega)\le m} \partial B_n}\le \Ocal\rk{\ex^{-c_1\frac{n}{m}}}\, .
    \end{equation}
    \end{lemma}
    \begin{proof}
        Write $\P\hk{\le m}$ for the loop soup which consists of loops of diameter less or equal to $m$, which can be obtained by restricting the loop measure $\mu$ to that set. Write $\dfrak(A,B)$ for the loop distance between two disjoint sets $A,B$:
        \begin{equation}
            \dfrak\rk{A,B}=\inf\{k\ge 1, \exists \omega_1,\ldots,\omega_k\colon \\A\cap\omega_1\neq \emptyset,\omega_1\cap\omega_2\neq \emptyset,\ldots,\omega_{k-1}\cap\omega_k\neq\emptyset,B\cap\omega_k\neq\emptyset\}\, .
        \end{equation}
        Write $\Ccal_0$ for the cluster intersecting the origin, $\Ccal_{\dfrak= k}$ for the subset of that cluster induced by loops $\omega$ with loop distance $k$ from the origin, analogously for $\Ccal_{\dfrak\le  k}$.

        Choose $k\in\N$ large enough, such that $\E\hk{\le m}\ek{\#\Ccal_{\dfrak= k}}\le \frac{1}{2}$. This is possible since $\alpha<\alpha_\#$ guarantees finite expected cluster size.

        For the loop soup sampled from $\P\hk{\le m}$, the following implication holds: for $L\in\N$ if the the event $\gk{0\lra \partial B_{2mk(L+1)}}$ occurs, then there must exist a loop in $\Ccal_{\dfrak=k}$ such that this loop is connected to the boundary of $B_{2mk(L+1)}$ through loops with a loop distance exceeding $k$ to the origin. Therefore
        \begin{equation}\label{eq:62420251}
            \P\hk{\le m}\rk{0\lra \partial B_{2mk(L+1)}}\le \P\hk{\le m}\rk{\exists x\in \Ccal_{\dfrak=k}\colon x\lro{{\dfrak>k}}\partial B_{2mk(L+1)}}\, .
        \end{equation}
        Under $\P\hk{\le m}$, $\Ccal_{\dfrak= k}$ has diameter at most $mk$. Hence, any $x$ as in Eq.~\eqref{eq:62420251} still has an Euclidean distance $2mkL$ to the boundary of $B_{2mk(L+1)}$. The following upper bound for $x\in \Ccal_{\dfrak=k}$ thus holds
        \begin{equation}
            \P\hk{\le m}\rk{ x\lro{{\dfrak>k}}\partial B_{2mk(L+1)}\mid \Ccal_{\dfrak\le k}}\le \P\hk{\le m}\rk{0\lra\partial B_{2mkL}}\, .
        \end{equation}
        Applying this estimate to bound the probability in Eq.~\eqref{eq:62420251} yields
        \begin{equation}
             \P\hk{\le m}\rk{0\lra \partial B_{2mk(L+1)}\mid 0\lra\partial B_{2mkL}}\le \sum_{x\in B_{mk}}\P\hk{\le m}\rk{x\in \Ccal_{\dfrak=k}}
             \le\E\hk{\le m}\ek{\#\Ccal_{\dfrak= k}}\leq \frac{1}{2} .
        \end{equation}
        Iterating this then yields
        \begin{equation}
             \P\hk{\le m}\rk{0\lra \partial B_{2mk(L+1)}}\le \rk{\frac{1}{2}}^{L}\, ,
        \end{equation}
        and hence the claim.
    \end{proof}
    The next lemma estimates the probability that two long loops are contained in the same cluster.
    \begin{lemma}\label{lem:twoLoops}
        There exists $c_2>0$ such that for all $m\ge 1$
        \begin{equation}
            \P\rk{\exists\omega_1\neq \omega_2\in \Ccal_0\colon \diam(\omega_1)>m,\, \diam(\omega_2)>m}\le c_2m^{6-2d}\, .
        \end{equation}
    \end{lemma}
    The exponent $6-2d$ should be thought of as $2\times (2-d)+2$; for each large loop $(2-d)$ and the additional $2$ for the range of each long loop.
    \begin{proof}
Since the loop soup is given by a Poisson point process
\begin{equation}
            \P\rk{\exists \omega \colon 0\lro{\omega}\partial B_m}=1-\ex^{-\alpha \mu\ek{0\lro{\omega}\partial B_m}}=\Ocal\rk{\alpha m^{2-d}}\, .
        \end{equation}
Similarly, the event that there are two long loops intersecting the origin has probability $\Ocal\rk{m^{4-2d}}=o\rk{m^{6-2d}}$ and is therefore negligible. Let $\Ccal_0\hk{\le m}$ be the open cluster at the origin formed by loops of diameter at most $m$ and $\Ccal\hk{\le m}_x$ the open cluster centered at $x$ formed by loops of diameter at most $m$. Recall that $\Ccal_0\hk{\le m}\neq \emptyset$, as it contains the origin. Set
        \begin{equation}
            A=\gk{\exists\omega_1\cap \Ccal_0\hk{\le m}\neq \emptyset\colon \diam(\omega_1)>m}\, .
        \end{equation}
        Define the (tree grown from $\omega_1$) following set
        \begin{equation}
            \Ccal_0\hk{\le m}\rk{\omega_1}=\Ccal_0\hk{\le m}\cup\bigcup_{x\in \Rcal(\omega_1)}\Ccal\hk{\le m}_x\, .
        \end{equation}
        The loop $\omega_1$ might not be unique, since two or more such loops may exist. But, one can make a unique choice of $\omega_1$ by the monotonicity of the Poisson point process and the finiteness of $\E\ek{\#\Ccal_0\hk{\le m}}$. Rewrite
        \begin{multline}\label{eq63020253}
             \P\rk{\exists\omega_1\neq \omega_2\in \Ccal_0\colon \diam(\omega_1)>m,\, \diam(\omega_2)>m}\\
              =\P\rk{A}\P\rk{\exists \omega_2\cap \Ccal_0\hk{\le m}\rk{\omega_1}\neq \emptyset\colon\diam(\omega_2)>m\mid A}\, .
        \end{multline}
        By Lemma~\ref{lem:singleLoopExact}
        \begin{equation}\label{eq63020254}
            \P\rk{A}\le \P\rk{0\lra \partial B_m} \le C  m^{2-d}\, .
        \end{equation}
        Furthermore, by independence properties of the Poisson point process and Lemma~\ref{lem:singleLoopExact}
        \begin{equation}\label{eq63020255}
            \P\rk{\exists \omega_2\cap \Ccal_0\hk{\le m}\rk{\omega_1}\neq \emptyset\colon\diam(\omega_2)>m\mid A}\le C m^{2-d}\E\ek{\#\Ccal_0\hk{\le m}\rk{\omega_1}\mid A}\, .
        \end{equation}
        By the disjointness of the loop sets
        \begin{equation}\label{eq63020256}
            \E\ek{\#\Ccal_0\hk{\le m}\rk{\omega_1}\mid A}\le \E\ek{\#\Ccal_0\hk{\le m}} \E\ek{\#\Rcal\rk{\omega_1}\mid A}\, .
        \end{equation}
        Lemma~\ref{ref:lemmasize} and the Mecke equation from Lemma \eqref{lem:mec} prove that $ \E\ek{\#\Rcal\rk{\omega_1}\mid A}=\Ocal\rk{m^{2}}$. Inserting this into Eq.~\eqref{eq63020256} yields
        \begin{equation}
             \E\ek{\#\Ccal_0\hk{\le m}\rk{\omega_1}\mid A}=\Ocal\rk{m^2}\, .
        \end{equation}
        Inserting the equation above into Eq.~\eqref{eq63020255} and then using Eq.~\eqref{eq63020253} concludes the proof.
    \end{proof}
\textbf{Proof of Theorem ~\ref{thm1}}. The main idea is as follows: rewrite
    \begin{equation}
          \P\rk{0\lra \partial B_n}=\P\rk{\exists \omega\colon \Ccal_0\rk{\eta\setminus \omega}\lro{\omega}\partial B_n}\, .
    \end{equation}    
    By the Mecke equation (see Lemma~\ref{lem:mec}), if this $\omega$ was unique, the above could be rewritten as
    \begin{equation}
        \alpha \int  \P\rk{\Ccal_0\rk{\eta}\lro{\omega}\partial B_n}\d \mu(\omega)\, .
    \end{equation}
    Using Fubini's theorem and Eq.~\eqref{eq:rep}, this equals
    \begin{equation}
        \alpha\int \sum_{y\in \Ccal_0}\sum_{N\ge 1}\frac{1}{N}\P_y\rk{\tau_{2N}<\infty, \omega\rk{\tau_{2N}}=y}\d \P\rk{\eta}\, ,
    \end{equation}
    with $\tau_0=0$, $\tau_{2N+1}=\inf\gk{t>\tau_{2N}\colon \omega(t)\in \partial B_n}$ and $\tau_{2N+2}=\inf\gk{t>\tau_{2N+1}\colon \omega(t)\in \Ccal_0}$, $N\ge 0$ (and $\Ccal_0$ not depending on $\omega$!). Assuming that $\Ccal_0(\eta)$ stays localized, Lemma~\ref{lem:singleLoopExact} gives
    \begin{equation}\label{eq:3320262}
        \sum_{N\ge 1}\frac{1}{N}\P_y\rk{\tau_{2N}<\infty, \omega\rk{\tau_{2N}}=y}\ge \P_y\rk{\tau_{2}<\infty, \omega\rk{\tau_{2}}=y}\sim \frac{C_d\Es_{\Ccal_0}(y)}{n^{d-2}}\, ,
    \end{equation}
    where we abbreviate $a_n=b_n(1+o(1))$ by $a_n\sim b_n$. Eq.~\eqref{eq:3320262} implies that
    \begin{equation}
         \P\rk{0\lra \partial B_n}\sim \frac{C_d}{n^{d-2}}  \alpha\int \sum_{y\in \Ccal_0(\eta)}\Es_{\Ccal_0}(y)\d \P\rk{\eta}=\alpha \frac{C_d\E\ek{\Cap\rk{\Ccal_0}}}{n^{d-2}}\, . 
    \end{equation}
    We now formalize this argument above. The lower bound is shorter: $\gk{0\lra \partial B_n}$ contains the event that there exists a unique loop $\omega$ connecting $\Ccal_0\cap B_k$ to $\partial B_n$, for $k>0$ fixed. By the previous discussion
    \begin{multline}
        \P\rk{0\lra \partial B_n}\ge \P\rk{\exists !\omega\textnormal{ such that }\Ccal_0(\eta\setminus\omega)\cap B_k\lro{\omega}\partial B_n}  \\
        \ge \alpha\int \sum_{y\in \Ccal_0\cap B_k}\P_y\rk{\tau_{2}<\infty, \omega\rk{\tau_{2}}=y}\1\gk{Q(\eta)}\d \P\rk{\eta}\, ,
    \end{multline}
    where $Q(\eta)$ is the event that $\eta$ does not already contain a loop connecting $\Ccal_0\cap B_k$ to $\partial B_n$.\\
    By \cite[Proposition 6.5.1]{lawler2010random} for $k$ fixed
    \begin{equation}
        \P_y\rk{\tau_{2}<\infty, \omega\rk{\tau_{2}}=y}= \frac{C_d\Es_{\Ccal_0\cap B_k}(y)}{n^{d-2}}(1+o(1))\, .
    \end{equation}
    This implies that for every $k$ fixed
    \begin{equation}
        \liminf_{n\to \infty}\frac{\P\rk{0\lra \partial B_n}}{n^{d-2}}\ge \alpha C_d\E\ek{\Cap\rk{\Ccal_0\cap B_K}\1\gk{Q(\eta)}}\, . 
    \end{equation}
    Taking the limit $k\to\infty$ proves the lower bound (via monotone convergence).
    
    The upper bound is established as follows: define $\nminus=\frac{n}{\log^2(n)}$. By Lemma~\ref{lem:shortLoopsConnect}, connecting the origin to $\partial B_n$ using only loops of diameter at most $\nminus$ has probability $\Ocal\rk{\ex^{-c\log^2(n)}}$, decaying faster than any polynomial in $n$. Hence
    \begin{equation}
        \P\rk{0\lra \partial B_n}= \P\rk{0\lra \partial B_n, \exists \omega_0\in \Ccal_0\colon \diam\rk{\omega_0}>\nminus}(1+o(1))\, . 
    \end{equation}
    Write
    \begin{equation}
       A(\nminus)=\gk{0\lra \partial B_n, \exists ! \omega_0\in \Ccal_0\colon \diam\rk{\omega_0}>\nminus}\, .
    \end{equation}
    We would like to show that
    \begin{equation}\label{eq:3420261}
        \P\rk{0\lra \partial B_n}\le  \P\rk{A(\nminus)}(1+o(1))\, .
    \end{equation}
   To establish \eqref{eq:3420261}, we apply Lemma~\ref{lem:twoLoops} to preclude the existence of multiple long loops. This lemma gives a decay of order $m^{6-2d}$ for the probability that there exist two or more long loops in $\Ccal_0$ of diameter exceeding $m$. If one chooses $m$ such that $m^{6-2d}$ is $o\rk{n^{2-d}}$, the occurrence of two large loops can be excluded. The function $d\mapsto\frac{d-2}{2d-6}$ achieves its maximum (on the domain $d\ge 5$) at $d=5$ with value $\frac{3}{4}$. Choosing $m=n^{5/6}$ yields $m^{6-2d}=\rk{n^{5/6}}^{6-2d}=o\rk{n^{2-d}}$. Hence, Lemma~\ref{lem:twoLoops} implies the event that any loops except $\omega_0$ have a diameter exceeding $n^{5/6}$ is negligible, proving Eq.~\eqref{eq:3420261}.
   
   Fix now $m=n^{5/6}$. By Lemma~\ref{lem:shortLoopsConnect}, the cluster grown from loops of diameter at most $m$ has to stay confined to $B_{n^{6/7}}$. Thus
    \begin{equation}
         \P\rk{0\lra \partial B_n}=\P\rk{\Ccal_0\hk{\le m}\cap \partial B_{n^{6/7}}=\emptyset \textnormal{ and } A(\nminus)}(1+o(1))\, .
    \end{equation}
In summary, the preceding arguments establish the existence of a loop of diameter at least $\nminus$ which enters the $n^{6/7}$-neighborhood of the origin.
    
Next, we prove that the distance between $\omega_0$ and $\partial B_n$ is $o(n)$. Indeed, suppose that $\dist(\omega_0,\partial B_n)>n^\e$ for some $\e>0$. Then, there needs to exist $x\in \Ccal_0\hk{\le m}\cup \Rcal\rk{\omega_0}$, such that $x\lro{\eta\setminus(\Ccal_0\hk{\le m}\cup  \omega_0)}\partial B_n$. The cost for the potential starting points of such $x$ is computed in Lemma~\ref{ref:lemmasize}:
    \begin{equation}
        \E\ek{\Ccal_0\hk{\le m}}\E\ek{\#\Rcal\rk{\omega_0}| A(\nminus)}\le \Ocal\rk{\nminus^2}\, ,
    \end{equation}
 The cost of connecting $x$ to $\partial B_n$ is $n^{\e (2-d)}$ by Lemma~\ref{lem:singleLoopExact}. If $\e>\frac{2}{d-2}\ge \frac{2}{3}$, it holds that
    \begin{equation}
        \nminus^{2-d}\nminus^2 n^{\e(2-d)}=o\rk{n^{2-d}}\, ,
    \end{equation}
    and hence $\dist(\omega_0,\partial B_n)<n^\e$ for such $\epsilon$, outside of a set of negligible probability.
    
    To summarize: there exists a unique $\omega_0\in \Ccal_0$ with $\partial B_{n^{6/7}}\lro{\omega_0}\partial B_{n-n^{\e}}$. This is enough to apply our approach
    \begin{equation}\label{eq:872025}
          \P\rk{0\lra \partial B_n} = \alpha\int \sum_{y\in \Ccal_0}\sum_{N\ge 1}\frac{1}{N}\P_y\rk{\tau_{2N}<\infty, \omega\rk{\tau_{2N}}=y}\1\gk{\diam(\Ccal_0)\le n^{6/7}}\d \P\rk{\eta}(1+o(1))\, ,
    \end{equation}
where $\tau_0=0$, $\tau_{2N+1}=\inf\gk{t>\tau_{2N}\colon \omega(t)\in \partial B_{n-n^\e}}$ and $\tau_{2N+2}=\inf\gk{t>\tau_{2N+1}\colon \omega(t)\in \Ccal_0}$, $N\ge 0$.\\
By Lemma~\ref{lem:singleLoopExact}, if $\diam(\Ccal_0)\le n^{6/7}$ then as $n\to \infty$
\begin{equation}
    \sum_{y\in \Ccal_0}\sum_{N\ge 1}\frac{1}{N}\P_y\rk{\tau_{2N}<\infty, \omega\rk{\tau_{2N}}=y}= \alpha\frac{C_d\Cap\rk{\Ccal_0}}{n^{d-2}}(1+o(1))\, .
\end{equation}
Since we only seek an upper bound, we can now remove the indicator function in Eq.~\eqref{eq:872025}. This implies that
\begin{equation}
      \P\rk{0\lra \partial B_n} \le  \alpha\int \frac{C_d\Cap\rk{\Ccal_0}}{n^{d-2}}\d \P\rk{\eta}\rk{1+o(1)}=\alpha\frac{C_d\E\ek{\Cap\rk{\Ccal_0}}}{n^{d-2}}\, .
\end{equation}
This concludes the proof of Eq.~\eqref{eq:thm1eq1}. To prove Eq.~\eqref{eq:thm1eq2}, observe that
\begin{equation}
    \lim_{k\to\infty}\E\ek{\Cap\rk{\Ccal_0},\Ccal_0\cap \partial B_k\neq \emptyset}=0\, ,
\end{equation}
using the definition of $\alpha_\#$ and the bound $\Cap(A)\le \#A$.
\subsection{Proof of Theorem 2}
The strategy is as follows: there must exist at least one very large loop inside $\Ccal_0$. That loop having a large distance to $x$ comes at a big probabilistic cost. If that loop is close to $x$, it can connect to $x$ using smaller, localized loops. Applying the Mecke equation and Lemma~\ref{lem:ConnectTwoSetsExact} yield the result.

We now make this rigorous and begin with a lemma which excludes the probability that geodesics connecting two points are large.
\begin{lemma}\label{lem:farconnect}
    If $\e>0$ and $m=\abs{x}^{1+\e}$, then as $\abs{x}\to\infty$
    \begin{equation}
    \P\rk{0\lra x\textnormal{ but not }0\lro{\subseteq B_m}x}=\Ocal\rk{\abs{x}^{(4-2d)(1+\e/4)}}=o\rk{\abs{x}^{4-2d}}\, .
\end{equation}
\end{lemma}
\begin{proof}
    Write $\Ccal_y\hk{\subseteq m}$ for the cluster at $y$ strictly contained in $B_m(y)$, $y\in \Z^d$. Write $A=\gk{0\lra x\textnormal{ but not }0\lro{\subseteq B_m}x}$, for the event in question. Choosing $\abs{x}=n$, $k=m^{1-\e/2}$, $M=n^{1-\delta}$ and $\delta<\e/2$ yields
\begin{equation}
    \P\rk{A}\le \P\rk{E_1}+\P\rk{E_2}+\P\rk{E_3}\, ,
\end{equation}
with
\begin{equation}
    \begin{split}
        E_1&=\gk{0\lra \partial B_{k}}\Box\gk{x\lra \partial B_{k}(x)}\, ,\\
        E_2&=\gk{\exists \omega\cap\partial B_m\neq \emptyset \textnormal{ and }\Ccal\hk{\subseteq k}_0\lro{\omega}\Ccal\hk{\subseteq k}_x,\, \dist\rk{\Ccal\hk{\subseteq k}_0,\Ccal\hk{\subseteq k}_x}>M}\, ,\\
        E_3&=\gk{\exists \omega\cap\partial B_m\neq \emptyset \textnormal{ and }\Ccal\hk{\subseteq k}_0\lro{\omega}\Ccal\hk{\subseteq k}_x,\, 1\le \dist\rk{\Ccal\hk{\subseteq k}_0,\Ccal\hk{\subseteq k}_x}\le M}\, .
    \end{split}
\end{equation}
\begin{figure}
    \centering
    \includegraphics[width=0.4\linewidth]{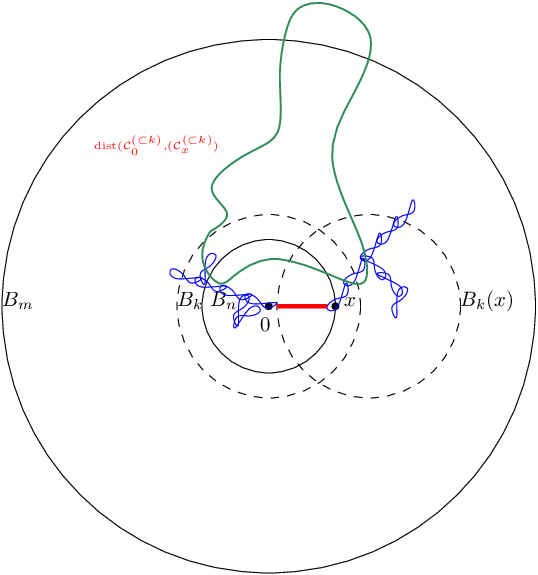}
    \caption{Illustration for Lemma~\ref{lem:farconnect}.}
    \label{fig:lemfar}
\end{figure}
Indeed, either $E_1$ or $E_1^c$ must occur. On $E_1^c$, there has to be a loop $\omega$ such that $\omega\cap\partial B_m\neq \emptyset$ and $\Ccal\hk{\subseteq k}_0\lro{\omega}\Ccal\hk{\subseteq k}_x$. If such a loop was not present in the soup, then either $\gk{0\lro{\subseteq B_m}x}$ or $E_1$.

By the BK inequality, $E_1$ can be neglected, as its probability is bounded by
\begin{equation}
    \P\rk{E_1}\le \P\rk{0\lra \partial B_{k}}^2=\Ocal\rk{k^{4-2d}}\, .
\end{equation}
To control $E_2$, rewrite it as follows
\begin{equation}
    \P\rk{E_2}= \P\rk{\exists a\in \Ccal\hk{\subseteq k}_0\textnormal{ and }b\in \Ccal\hk{\subseteq k}_x\colon a\lro{\omega}b\textnormal{ and }\omega\cap \partial B_m\neq \emptyset\textnormal{ and }\dist(a,b)>M }\, .
\end{equation}
The above can bounded by
\begin{equation}
    \sum_{\mycom{a\in B_k}{b\in B_k(x)}}\P\rk{\gk{0\lro{\subseteq  B_m } a}\Box \gk{x\lro{\subseteq  B_m }b}\cap \gk{\exists \omega \colon a\lro{\omega}b\textnormal{ and }\omega\cap \partial B_m\neq \emptyset}}\1\gk{\dist(a,b)>M}\, .
\end{equation}

Adapting Lemma~\ref{lem:ConnectTwoSetsExact} shows that for $\dist(a,b)>M$
\begin{equation}
    \mu\ek{a\lro{\omega}b\textnormal{ and }\omega\cap \partial B_m\neq \emptyset}=\Ocal\rk{m^{2-d}M^{2-d}}=\Ocal\rk{n^{(2+\e-\delta)(2-d)}}\, ,
\end{equation}
as the connection from $\partial B_m$ to $b$ costs $m^{2-d}$ and from $b$ to $a$ costs $M^{2-d}$, see Lemma~\ref{lem:largesingkeloop}. 

By first conditioning on $ \Ccal\hk{\subseteq k}_0$ and $ \Ccal\hk{\subseteq k}_x$ can hence bound
\begin{equation}
    \P\rk{E_2}\le \Ocal\rk{n^{(2-d)(2+\e/2)}} \sum_{\mycom{a\in B_k}{b\in B_k(x)}}\P\rk{\gk{0\lra a}\Box \gk{x\lra b}}\le  \Ocal\rk{n^{(2-d)(2+\e/2)}\E\ek{\Ccal_0}^2}\, ,
\end{equation}
by the BK-inequality.

We now claim that
\begin{equation}
    \P\rk{E_3}\le 2\P\rk{\gk{0\lra \partial B_{n/3}}\Box \gk{\exists b\in \Ccal\hk{\subseteq k}_x\textnormal{ and }\omega\colon b\lro{\omega}\partial B_m}}\, .
\end{equation}
Indeed, $\dist\rk{\Ccal\hk{\subseteq k}_0,\partial B_{n/3}}>0$ and $\dist\rk{\Ccal\hk{\subseteq k}_x,\partial B_{n/3}(x)}>0$ cannot be satisfied simultaneously, as this would imply that $\dist\rk{\Ccal\hk{\subseteq k}_0,\Ccal\hk{\subseteq k}_x}>n/3$ and hence violate $\dist\rk{\Ccal\hk{\subseteq k}_0,\Ccal\hk{\subseteq k}_x}\le M=o(n)$. Suppose that $\dist\rk{\Ccal\hk{\subseteq k}_0,\partial B_{n/3}}=0$. By the condition $\Ccal\hk{\subseteq k}_0\lro{\omega}\Ccal\hk{\subseteq k}_x$, there has to exist $b\in \Ccal\hk{\subseteq k}_x$ with $b\lro{\omega}\partial B_m$. This implies the bound (using the BK inequality first and then Lemma~\ref{lem:largesingkeloop})
\begin{equation}
    \begin{split}
        \P\rk{E_3}&\le 2\P\rk{0\lra \partial B_{n/3}}\P\rk{\exists b\in \Ccal\hk{\subseteq k}_x\colon b\lro{\omega}\partial B_m}\\
        &\le C n^{2-d}\sum_{y\in B_k(x)}\P\rk{y\in \Ccal\hk{\subseteq k}_x,y\lro{\textnormal{not } \subseteq B_k(x)} \partial B_m}\\
        &\le C n^{2-d}\dist(B_k(x),\partial B_m)^{2-d}\E\ek{\#\Ccal\hk{\subseteq k}_x}\\
        &=\Ocal\rk{n^{(4-2d)(1+\e/2)}}\, .
    \end{split}
\end{equation}
This concludes the proof.
\end{proof}
\textbf{Lower bound of Theorem~\ref{thm:twoPoint}}. For $k\ge 1$ fixed 
\begin{equation}
    \P\rk{0\lra x}\ge \P\rk{\exists!\omega\colon \Ccal_0\rk{\eta\setminus \omega}\subseteq B_k\textnormal{ and }\Ccal_x\rk{\eta\setminus \omega}\subseteq B_k(x)\textnormal{ and } \Ccal_0\rk{\eta\setminus \omega}\lro{\omega }\Ccal_x\rk{\eta\setminus \omega}}.
\end{equation}
Using the Mecke equation, the above can be rewritten as
\begin{equation}\label{eq:81220251}
   \alpha \int\d \P(\eta)\mu\ek{ \Ccal_0\rk{\eta}\lro{\omega }\Ccal_x\rk{\eta}}\1\gk{\diam\rk{\Ccal_0}\le k,\, \diam\rk{\Ccal_x}\le k}\, .
\end{equation}
By Lemma~\ref{lem:ConnectTwoSetsExact}, the above is equal to up to first order
\begin{equation}
   \alpha  C_d^2\abs{x}^{4-2d}\int\d \P(\eta)\Cap\rk{\Ccal_0}\Cap\rk{\Ccal_x}\1\gk{\diam\rk{\Ccal_0}\le k,\, \diam\rk{\Ccal_x}\le k}\, .
\end{equation}
Taking the limit $k\to\infty$ gives
\begin{equation}
    \lim_{\abs{x}\to\infty}\abs{x}^{2d-4}\P\rk{0\lra x}\ge \alpha C_d^2 \lim_{\abs{x}\to\infty}\E\ek{\Cap\rk{\Ccal_0}\Cap\rk{\Ccal_x}}\ge \alpha C_d^2 \E\ek{\Cap\rk{\Ccal_0}}^2\, ,
\end{equation}
where the last step follows from translation invariance and the FKG-inequality.

\textbf{Upper bound of Theorem~\ref{thm:twoPoint}}. Write $n=\abs{x}$ for convenience.

\textbf{Case 1:} assume that for $m=\frac{n}{\log(n)^4}$
\begin{equation}
    \P\rk{0\lra x}= \P\rk{x\in\Ccal_0\textnormal{ and }\exists ! \omega\in \Ccal_0\colon \diam(\omega)>m}(1+o(1))\, .
\end{equation}
Let $\Ccal\hk{\le m}_0$ be the cluster centered at the origin consisting of loops $\omega$ of diameter $\diam(\omega)\le m$. The inclusion $\Ccal\hk{\le m}_0\subseteq B_M$ holds outside a set of negligible probability, for $M=\frac{n}{\log(n)^2}$, on account of Lemma~\ref{lem:shortLoopsConnect} ($\Ocal\rk{\ex^{-cM/m}}=\Ocal\rk{\ex^{-c\log^2(n)}}$ decays faster than any polynomial in $n$).

For $k<n-M$, we can exclude the case
\begin{equation}
    \gk{x\in\Ccal_0\textnormal{ and }\exists ! \omega\in \Ccal_0\colon \diam(\omega)>m\textnormal{ and }\dist(x,\omega)>k}\, ,
\end{equation}
as in that case there would exist a $y\in \Ccal_0\hk{\le m}\cup\Rcal\rk{\omega}$ that connects to $x$. Such $y$ would necessarily have to travel a distance of $\min\gk{k,n-M}=k$ to connect to $x$ which costs $k^{4-2d}$. Choose $k=k_\gamma=m^{\frac{\gamma}{2d-4}}$. The probability of the event above is bounded (again, using the BK inequality as in the previous lemma) by
\begin{equation}
    \Ocal\rk{m^{2-d}m^2 k^{4-2d}}=\Ocal\rk{m^{4-2d-\gamma}}=o\rk{n^{4-2d}}\, ,
\end{equation}
where the last equality is true if $\gamma>d$ (and therefore $\gamma/\rk{2d-4}>\frac{5}{6}$). Hence, there exists a unique loop $\omega$ connecting $\Ccal\hk{\le m}_0$ and $\Ccal_x\hk{\le m}\subseteq B_M(x)$. By symmetry $\dist(0,\omega)\le k$. We then can conclude as before: using the Mecke equation and Lemma~\ref{lem:ConnectTwoSetsExact}, expand as in Eq.~\eqref{eq:81220251}
\begin{equation}
    \P\rk{\exists ! \omega\colon\Ccal_0\hk{\subseteq M}\lro{\omega}\Ccal_x\hk{\subseteq M}}= \frac{C_d^2}{n^{4-2d}}\E\ek{\Cap\rk{\Ccal_0\hk{\subseteq M}}\Cap\rk{\Ccal_x\hk{\subseteq M}}}\rk{1+o(1)}\, .
\end{equation}
Since $\Cap\rk{\Ccal_0\hk{\subseteq M}}$ and $\Cap\rk{\Ccal_x\hk{\subseteq M}}$ are independent
\begin{equation}
    \E\ek{\Cap\rk{\Ccal_0\hk{\subseteq M}}\Cap\rk{\Ccal_x\hk{\subseteq M}}}=\E\ek{\Cap\rk{\Ccal_0\hk{\subseteq M}}}^2\le \E\ek{\Cap\rk{\Ccal_0}}^2\, .
\end{equation}

\textbf{Case 2}: there must exist at least one $\omega$ of diameter exceeding $m$, on account of Lemma~\ref{lem:shortLoopsConnect}. It therefore remains to exclude the possibility of two or more such $\omega$'s existing. Denote two large $\omega$'s by $\omega_1,\omega_2$. As shown in Case 1, at least one of them has to connect $\Ccal_0\hk{\subseteq M}$ to $B_k(x)$. Assume that $\dist(x,\omega_1)\le k$. Fix $\e>0$. Choose $\gamma>d$ such that $m/k>k^{1/6}$ (i.e., $\gamma=d+0.01$) to obtain
\begin{equation}
    \frac{m}{k}=k^{\frac{2d-4-\gamma}{\gamma}}>k^{1/6}\, .
\end{equation}
By Lemma~\ref{lem:farconnect}, for any $y\in B_k(x)$ the probability that $y\lra x$ but not $y\lro{\subseteq B_M(x)}x$ has probability $\Ocal\rk{k^{(4-2d)(1+1/25)}}$. Hence
\begin{equation}\label{eq:31220261}
    \P\rk{0\lra x\textnormal{ but not }\Ccal_x\hk{\subseteq k}\lro{\Ccal_x\hk{\subseteq M}}\Ccal_0\hk{\subseteq M}}=\Ocal\rk{m^{4-2d}k^{d-2}m^2k^{(4-2d)(1+1/25)}}\, ,
\end{equation}
where $m^{4-2d}k^{d-2}$ is for the connection from $\Ccal_0\hk{\subseteq M}$ to $B_k(x)$ via $\omega_1$ (see Lemma~\ref{lem:ConnectTwoSetsExact}), $m^2$ is for the expected range of $\omega_1\cap B_k(x)$ (see Lemma~\ref{ref:lemmasize}) and $k^{(4-2d)(1+1/25)}$ is the price of needing a far-away loops to connect $y$ to $x$. Comparison of exponents yields
\begin{equation}
    k^{d-2}m^2k^{(4-2d)(1+1/25)}=m^2m^{-\frac{\gamma}{2}(1+1/25)}=\Ocal\rk{m^{-1/25}}\, ,
\end{equation}
and therefore the probability in Eq.~\eqref{eq:31220261} can bound by $\Ocal\rk{m^{(4-2d)-1/25}}=o\rk{\P\rk{0\lra x}}$.

Thus, the connection from $0$ to $x$ does not need $\omega_2$ and hence
\begin{equation}
    \P\rk{0\lra x}= \P\rk{\exists! \omega \colon \Ccal_0\hk{\le m}\lro{\omega}\Ccal\hk{\le m}_x,\, \Ccal_0\hk{\le m}\subset B_M\textnormal{ and }\Ccal_x\hk{\le m}\subset B_M(x)}\rk{1+o(1)}\, .
\end{equation}
This shows that Case 1 dominates. Hence, the proof of the upper bound is finished.
\section{Acknowledgments}
The author thanks the anonymous referees for their many helpful suggestions and comments. This helped to significantly improve the paper. The author would also like to thank Silke Rolles, for her support and encouragement.
\bibliographystyle{alpha}
\bibliography{thoughts.bib}
\end{document}

%% file: packages.tex
\usepackage[colorinlistoftodos]{todonotes}
\usepackage{pgfplots} 
\usepackage{subfigure}
\usepackage{amsmath}
\usepackage{amssymb} 
\usepackage{mathtools}   
\usepackage{graphicx}
\usepackage{graphics}  
\usepackage{color}
\usepackage{verbatim} 
\usepackage{bbold}
\usepackage[normalem]{ulem} 
\usepackage[mathscr]{eucal}
\usepackage{upgreek}
\usepackage{enumerate} 
\usepackage{cite}
\usepackage{amsmath}
\usepackage{amsfonts}
\usepackage{amssymb}
\usepackage{bbm}
\usepackage{cancel}
\usepackage{graphicx}
\usepackage{multicol}
\usepackage[utf8]{inputenc}
\usepackage{framed}
\usepackage{setspace}
\usepackage{amsthm}
\usepackage{hyperref}
\usepackage{caption}
\usepackage{subcaption}
\usepackage{bbm}
\usepackage{graphicx}
\usepackage{tikz}
\usepackage[titletoc]{appendix}
\numberwithin{equation}{section}  

\usepackage[refpage,noprefix]{nomencl}
\usepackage{nomencl} 
\usepackage{soul,xcolor}

%% file: commands.tex
\newcommand{\ssup}[1] {{\scriptscriptstyle{{#1}}}}
\newcommand{\gk}[1]{\left\{#1\right\}}

\newcommand\mycom[2]{\genfrac{}{}{0pt}{}{#1}{#2}}
\newcommand{\ek}[1]{\left[#1\right]}
\newcommand{\rk}[1]{\left(#1\right)}

\newcommand{\hk}[1]{^{\ssup{(#1)}}}
\newcommand{\abs}[1]{\left| #1 \right|}

\newcommand{\Ccal}   {{\mathcal C }}

\newcommand{\Ocal}   {{\mathcal O }}

\newcommand{\Rcal}   {{\mathcal R }}

\newcommand{\B}     {\mathbb{B}}

\newcommand{\Z}     {\mathbb{Z}} 
\newcommand{\N}     {\mathbb{N}} 
\renewcommand{\P}   {\mathbb{P}} 
 
\newcommand{\E}     {\mathbb{E}}

 \newcommand{\ex}{{\rm e}} 
 \renewcommand{\d}{{\rm d}}

\newcommand{\e}{\varepsilon}
\newcommand{\1}{\mathbbm{1}}

\renewcommand{\P}{\mathbb{P}}

\newtheorem*{remark}{Remark}

\newtheorem*{lemma*}{Lemma}

\newtheorem{theorem}{Theorem}[subsection]
\newtheorem{lemma}[theorem]{Lemma}

\newcommand{\diam}{\mathrm{diam}}

